\documentclass[12pt]{article}
\usepackage{amsmath,amsthm,amssymb,bbm,extsizes}
\usepackage{mathabx}
\usepackage[all]{xy}
\usepackage[active]{srcltx}
\usepackage[russian,ukrainian,english]{babel}
\usepackage[cp1251]{inputenc}

\DeclareMathOperator{\im}{Im}
\DeclareMathOperator{\Ker}{Ker}
\renewcommand{\le}{\leqslant}
\renewcommand{\ge}{\geqslant}

\newtheorem{theorem}{Theorem}
\newtheorem{lemma}{Lemma}

\theoremstyle{definition}
\newtheorem{definition}{Definition}

\begin{document}
\title{Topological classification
of oriented cycles of linear
mappings\thanks{Printed in
\emph{Ukrainian Math. J.}, 2014. The
second author was supported in part by
the Foundation for Research Support of
the State of S\~ao Paulo (FAPESP),
grant 2012/18139-2.}}

\author{Tetiana Rybalkina
\qquad Vladimir V. Sergeichuk\\
rybalkina\_t@ukr.net\qquad
sergeich@imath.kiev.ua\\ Institute of
Mathematics, Tereshchenkivska 3, Kiev,
Ukraine }
\date{}

\maketitle

\begin{abstract}
We consider oriented cycles
\[
\xymatrix{
{\mathbb F^{m_1}}
\ar@{<-}@/_2pc/[rrrr]^{A_t}
\ar@{->}[r]^{A_1}&
\mathbb F^{m_2}\ar@{->}[r]^{A_2\ \ } &
{\ \dots\ }&{\mathbb F^{m_{t-1}}}
\ar@{<-}[l]_{A_{t-2}}
\ar@{->}[r]^{\ A_{t-1}\ }&{\mathbb F^{m_t}}}
\]
of linear mappings over $\mathbb
F=\mathbb C\text{ or }\mathbb R$, and
reduce the problem of their
classification up to homeomorphisms in
the spaces $\mathbb
F^{m_1},\dots,\mathbb F^{m_t}$ to the
case $t=1$, which was studied by N.H.
Kuiper and J.W. Robbin [Invent.
Math.~19 (no.\,2) (1973) 83--106] and
by other authors.
\medskip

\end{abstract}

\section{Introduction}

We consider the problem of topological
classification of oriented cycles of
linear mappings.

Let
\begin{equation}\label{jsttw}
\begin{split}
{\cal A}: \qquad\xymatrix{
{V_1}
\ar@{<-}@/_2pc/[rrrr]^{A_t}
\ar@{->}[r]^{A_1}&
{V_2}\ar@{->}[r]^{A_2\ \ } &
{\ \dots\ }&{{V_{t-1}}}
\ar@{<-}[l]_{A_{t-2}}
\ar@{->}[r]^{\ A_{t-1}\ }
&{{V_t}}}
\end{split}
\end{equation}
and
\begin{equation}\label{jtw}
\begin{split}
{\cal B}: \qquad\xymatrix{
{W_1}
\ar@{<-}@/_2pc/[rrrr]^{B_t}
\ar@{->}[r]^{ B_1}&
{W_2}\ar@{->}[r]^{B_2\ \ } &
{\ \dots\ }&{{W_{t-1}}}
\ar@{<-}[l]_{B_{t-2}}
\ar@{->}[r]^{\ B_{t-1}\ }
&{{W_t}}}
\end{split}
\end{equation}
be two oriented cycles of linear
mappings of the same \emph{length $t$}
over a field $\mathbb F$. We say that a
system $\varphi=\{\varphi_i:V_i\to
W_i\}_{i=1}^t$ of bijections
\emph{transforms $\cal A$ to $\cal B$}
if all squares in the diagram
\begin{equation}\label{kjlj}
\begin{split}
\xymatrix@R=35pt{
{V_1}
\ar@{->}[d]^{\varphi_1}
\ar@{<-}@/_1.5pc/[rrrr]^{A_t}
\ar@{->}[r]^{A_1}
   & {V_2}
\ar@{->}[d]^{\varphi_2}
\ar@{->}[r]^{A_2\ \ } & {\ \dots\
}
   &{V_{t-1}}
\ar@{->}[d]^{\varphi_{t-1}}
\ar@{<-}[l]_{A_{t-2}} \ar@{->}[r]^{\
A_{t-1}\ }
   &{V_t}
\ar@{->}[d]^{\varphi_t}   \\
W_1
\ar@{<-}@/_1.5pc/[rrrr]^{B_t}
\ar@{->}[r]^{ B_1}&
{W_2}\ar@{->}[r]^{B_2\ \ } &
{\ \dots\ }&{W_{t-1}}
\ar@{<-}[l]_{B_{t-2}}
\ar@{->}[r]^{\ B_{t-1}\ }
&{W_t}}
\end{split}
\end{equation}
are commutative; that is,
\begin{equation}\label{kye}
\varphi_2A_1=B_1\varphi_1,\ \ \dots,\ \
\varphi_tA_{t-1}=B_{t-1}\varphi_{t-1},
\ \ \varphi_1A_t=B_t\varphi_t.
\end{equation}

\begin{definition}\label{def}
Let $\cal A$ and $\cal B$ be cycles of
linear mappings of the form
\eqref{jsttw} and \eqref{jtw} over a
field $\mathbb F$.

\begin{itemize}
  \item[\rm(i)] $\cal A$ and $\cal
      B$ are \emph{isomorphic} if
      there exists a system of
      linear bijections that
      transforms $\cal A$ to $\cal
      B$.

  \item[\rm(ii)] $\cal A$ and $\cal
      B$  are \emph{topologically
      equivalent} if $\mathbb
      F=\mathbb C$ or $\mathbb R$,
\begin{equation*}\label{j8o}
V_i=\mathbb F^{m_i},\quad
W_i=\mathbb F^{n_i}\quad
\text{for all }i=1,\dots,t,
\end{equation*}
and there exists a system of
homeomorphisms\footnote{By
\cite[Corollary 19.10]{Bred} or
\cite[Section 11]{McCl},
$m_1=n_1,\dots,m_t=n_t$.} that
transforms $\cal A$ to $\cal B$.
\end{itemize}
\end{definition}

The {\it direct sum} of cycles
\eqref{jsttw} and \eqref{jtw} is the
cycle
\begin{equation*}\label{zjmtw}
\begin{split}
{\cal A}\oplus{\cal B}: \quad
\xymatrix@C=15mm{ {V_1\oplus W_1}
\ar@{<-}@/_2pc/[rrr]^{A_t\oplus B_t}
\ar@{->}[r]^{A_1\oplus B_1}&
{V_2\oplus W_2}
\ar@{->}[r]^{A_2\oplus B_2} & {\ \dots\
}
\ar@{->}[r]^{A_{t-1}\oplus B_{t-1}\ \
} &{{V_t\oplus W_t}}}
\end{split}
\end{equation*}

The vector $ \dim {\cal A}:=(\dim
V_1,\dots,\dim V_t) $ is the
\emph{dimension} of $\cal A$. A cycle
$\cal A$ is \emph{indecomposable} if
its dimension is nonzero and $\cal A$
cannot be decomposed into a direct sum
of cycles of smaller dimensions.

A cycle $\cal A$ is \emph{regular} if
all $A_1,\dots,A_t$ are bijections, and
\emph{singular} otherwise. Each cycle
$\cal A$ possesses a \emph{regularizing
decomposition}
\begin{equation}\label{pyj}
\mathcal A=\mathcal A_{\mathrm{reg}}\oplus \mathcal A_1
\oplus\dots\oplus\mathcal A_r,
\end{equation}
in which $\mathcal A_{\mathrm{reg}}$ is
regular and all $\mathcal
A_1,\dots,\mathcal A_r$ are
indecomposable singular. An algorithm
that constructs a regularizing
decomposition of a nonoriented cycle of
linear mappings over $\mathbb C$ and
uses only unitary transformations was
given in \cite{ser-cycle}.

The following theorem  reduces the
problem of topological classification
of oriented cycles of linear mappings
to the problem of topological
classification of linear operators.

\begin{theorem}\label{yrw}
{\rm(a)} Let $\mathbb F=\mathbb C$ or
$\mathbb R$, and let
\begin{equation}\label{ktw}
\begin{split}
{\cal A}: \qquad\xymatrix{
{\mathbb F^{m_1}}
\ar@{<-}@/_2pc/[rrrr]^{A_t}
\ar@{->}[r]^{A_1}&
{\mathbb F^{m_2}}\ar@{->}[r]^{A_2\ \ } &
{\ \dots\ }&{{\mathbb F^{m_{t-1}}}}
\ar@{<-}[l]_{A_{t-2}}
\ar@{->}[r]^{\ A_{t-1}\ }
&{\mathbb F^{m_t}}}
\end{split}
\end{equation}
and
\begin{equation}\label{kjtw}
\begin{split}
{\cal B}: \qquad\xymatrix{
\mathbb F^{n_1}
\ar@{<-}@/_2pc/[rrrr]^{B_t}
\ar@{->}[r]^{ B_1}&
{\mathbb F^{n_2}}\ar@{->}[r]^{B_2\ \ } &
{\ \dots\ }&{\mathbb F^{n_{t-1}}}
\ar@{<-}[l]_{B_{t-2}}
\ar@{->}[r]^{\ B_{t-1}\ }
&{\mathbb F^{n_t}}}
\end{split}
\end{equation}
be topologically equivalent. Let
\begin{equation}\label{groj}
\mathcal A=\mathcal A_{\mathrm{reg}}\oplus \mathcal A_1
\oplus\dots\oplus\mathcal A_r,
\qquad \mathcal B=\mathcal B_{\mathrm{reg}}\oplus
\mathcal B_1
\oplus\dots\oplus\mathcal B_s
\end{equation}
be their regularizing decompositions.
Then their regular parts $\mathcal
A_{\mathrm{reg}}$ and $\mathcal
B_{\mathrm{reg}}$ are topologically
equivalent, $r=s$, and after a suitable
renumbering their indecomposable
singular summands $\mathcal A_i$ and
$\mathcal B_i$ are isomorphic for all
$i=1,\dots,r$.

{\rm(b)} Each regular cycle $\cal A$ of
the form \eqref{ktw} is isomorphic to
the cycle
\begin{equation}\label{tmtw}
\begin{split}
{\cal A}': \qquad\xymatrix{
{\mathbb F^{m_1}}
\ar@{<-}@/_2pc/[rrrr]^{A_t\cdots A_2A_1}
\ar@{->}[r]^{\mathbbm 1}&
{\mathbb F^{m_2}}\ar@{->}[r]^{\mathbbm 1\ \ } &
{\ \dots\ }&{\mathbb F^{m_{t-1}}}
\ar@{<-}[l]_{\mathbbm 1}
\ar@{->}[r]^{\ \mathbbm 1\ }
&{\mathbb F^{m_t}}}
\end{split}
\end{equation}
If cycles \eqref{ktw} and \eqref{kjtw}
are regular, then they are
topologically equivalent if and only if
the linear operators ${
{A}_t\cdots{A}_2{A}_1}$ and ${
{B}_t\cdots{B}_2{B}_1}$ are
topologically equivalent $($as the
cycles ${{\mathbb
F^{m_1}}\!\!\righttoleftarrow
\!
{A}_t\cdots{A}_2{A}_1}$ and ${{\mathbb
F^{n_1}}\!\!\righttoleftarrow\!
{B}_t\cdots{B}_2{B}_1}$ of length $1)$.
\end{theorem}

Kuiper and Robbin~\cite{Robb,Kuip-Robb} gave a criterion for topological
equivalence of linear operators over
$\mathbb R$ without eigenvalues that
are roots of~$1$. Budnitska
\cite[Theorem 2.2]{bud1} found a
canonical form with respect to
topological equivalence of linear
operators over $\mathbb R$ and $\mathbb
C$ without eigenvalues that are roots
of $1$. The problem of topological
classification of linear operators with
an eigenvalue that is a root of $1$ was
studied by Kuiper and
Robbin~\cite{Robb,Kuip-Robb}, Cappell
and Shaneson~\cite{Capp-conexamp,
Capp-2th-nas-n<=6,Capp-big-n<6,Cap+sha,Cap+ste}, and
Hsiang and Pardon~\cite{Pardon}. The
problem of topological classification
of affine operators was studied in
\cite{bud1,Ephr,Blanc,bud,bud+bud}. The
topological classifications of pairs of
counter mappins $V_1\,\begin{matrix}
      \longrightarrow \\[-4.1mm]
      \longleftarrow
    \end{matrix}\,V_2$ (i.e.,
oriented cycles of length $2$) and of
chains of linear mappings were given in
\cite{ryb_new} and \cite{ryb+ser}.


\section{Oriented cycles of linear
mappings up to isomorphism}\label{jdt}

This section is not topological; we
construct a regularizing decomposition
of an oriented cycle of linear mappings
over an arbitrary field $\mathbb F$.

A classification of cycles of length 1
(i.e., linear operators $
 {V}\!\righttoleftarrow$) over any field is given
by the Frobenius canonical form of a
square matrix under similarity. The
oriented cycles of length 2 (i.e.,
pairs of counter mappins
$V_1\,\begin{matrix}
      \longrightarrow \\[-4.1mm]
      \longleftarrow
    \end{matrix}\,V_2$) are classified in
\cite{dob+pon,hor+mer}. The
classification of cycles of arbitrary
length and with arbitrary
orientation of its arrows is well known
in the theory of representations of
quivers; see \cite[Section
11.1]{gab+roi}.

For each $c\in \mathbb Z$, we denote by
$[c]$ the natural number such that
\begin{equation*}\label{jit}
1\le [c]\le t,\qquad [c]\equiv c\bmod t.
\end{equation*}

By the Jordan theorem, for each
indecomposable singular cycle $
 {V}\!\righttoleftarrow
 {\! A}$
there exists a basis $e_1,\dots,e_n$ of
$V$ in which the matrix of $A$ is a
singular Jordan block. This means that
the basis vectors form a {\it Jordan
chain}
\[
e_1\xrightarrow{A} e_{2}
\xrightarrow{A}e_{3}
\xrightarrow{A}\cdots
\xrightarrow{A}e_n
\xrightarrow{A}0.
\]

In the same manner, each indecomposable
singular cycle $\cal A$ of an arbitrary
length $t$ also can be given by a chain
\begin{equation*}\label{griw}
e_p\xrightarrow{A_p} e_{p+1}
\xrightarrow{A_{[p+1]}}e_{p+2}
\xrightarrow{A_{[p+2]}}\cdots
\xrightarrow{A_{[q-1]}}e_q
\xrightarrow{A_{[q]}}0
\end{equation*}
in which $1\le p\le q\le t$ and for
each $l=1,2,\dots,t$ the set
$\{e_i|i\equiv l\bmod t\}$ is a basis
of $V_l$; see \cite[Section
11.1]{gab+roi}. We say that this chain
\emph{ends in $V_{[q]}$} since $e_q\in
V_{[q]}$. The number $q-p$ is called
the \emph{length} of the chain.

For example, the chain
\[
\xymatrix@R=7pt@C=11pt{
&&&e_4\ar[r]&e_5
\ar[dllll]\\
e_6\ar[r]&e_7
\ar[r]&e_8\ar[r]&e_9\ar[r]&e_{10}
\ar[dllll]\\
e_{11}\ar[r]&e_{12}
\ar[r]&0}
\]
of length 8 gives an indecomposable
singular cycle on the spaces $
V_1=\mathbb Fe_6\oplus\mathbb Fe_{11}$,
$V_2=\mathbb Fe_7\oplus\mathbb
Fe_{12},$ $V_3=\mathbb Fe_8$,
$V_4=\mathbb Fe_4\oplus\mathbb Fe_{9},$
$V_5=\mathbb Fe_5\oplus\mathbb
Fe_{10}$.

\begin{lemma}\label{yrwe}
Let
\begin{equation*}\label{js7w}
\begin{split}
{\cal A}: \qquad\xymatrix{
{V_1}
\ar@{<-}@/_2pc/[rrrr]^{A_t}
\ar@{->}[r]^{A_1}&
{V_2}\ar@{->}[r]^{A_2\ \ } &
{\ \dots\ }&{{V_{t-1}}}
\ar@{<-}[l]_{A_{t-2}}
\ar@{->}[r]^{\ A_{t-1}\ }
&{{V_t}}}
\end{split}
\end{equation*}
be an oriented cycle of linear
mappings, and let \eqref{pyj} be its
regularizing decomposition.

\begin{itemize}
  \item[\rm(a)] Write
\begin{equation}\label{jyf}
\hat{A}_i:=A_{[i+t-1]}\cdots A_{[i+1]}A_i:
V_i\to V_i
\end{equation}
and fix a natural number $z$ such
that
      \[
\tilde V_i:=
\hat A_i^zV_i=
\hat A_i^{z+1}V_i\quad\text{for all $i=1,\dots,t$.}
      \]
Let
\begin{equation*}\label{jss}
\begin{split}
\tilde{\cal A}: \qquad\xymatrix{
{\tilde V_1}
\ar@{<-}@/_2pc/[rrrr]^{\tilde A_t}
\ar@{->}[r]^{\tilde A_1}&
{\tilde V_2}\ar@{->}[r]^{\tilde A_2\ \ } &
{\ \dots\ }&{{\tilde V_{t-1}}}
\ar@{<-}[l]_{\tilde A_{t-2}}
\ar@{->}[r]^{\ \tilde A_{t-1}\ }
&{{\tilde V_t}}}
\end{split}
\end{equation*}
be the cycle formed by the
restrictions $\tilde A_i:\tilde
V_i\to\tilde V_{[i+1]}$ of
$A_i:V_i\to V_{[i+1]}$. Then ${\cal
A}_{\mathrm{reg}}=\tilde {\cal A}$
$($and so the regular part is
uniquely determined by $\cal A)$.

  \item[\rm(b)] The numbers
\[
k_{ij}:=\dim\Ker
      (A_{[i+j]}\cdots
      A_{[i+1]}A_i),\quad\text{
      $i=1,\dots,t$ and $j\ge 0$},
\]
determine the singular summands
$\mathcal A_1,\dots,\mathcal A_r$
of regularizing decomposition
\eqref{pyj} up to isomorphism since
the number $n_{lj}$ $(l=1,\dots,t$
and $j\ge 0)$ of singular summands
given by chains of length $j$ that
end in $V_l$ can be calculated by
the formula
\begin{equation}\label{byt}
n_{lj}=
k_{[l-j],j}-k_{[l-j],j-1}-
k_{[l-j-1],j+1}+k_{[l-j-1],j}
\end{equation}
in which $k_{i,-1}:=0$.
\end{itemize}
\end{lemma}

\begin{proof}
(a) Let \eqref{pyj} be a regularizing
decomposition of $\cal A$. Let
\[V_i=V_{i,{\mathrm{reg}}}\oplus V_{i1}
\oplus\dots\oplus V_{ir},\qquad
i=1,\dots,t,\] be the corresponding
decompositions of its vector spaces.
Then $\hat
A_i^zV_{i,\text{reg}}=V_{i,\text{reg}}$
(since all linear mappings in ${\cal
A}_{\mathrm{reg}}$ are bijections) and
$\hat A_i^zV_{i1}=\dots=\hat
A_i^zV_{ir}=0$. Hence
$V_{i,\text{reg}}=\tilde V_i$, and so
${\cal A}_{\mathrm{reg}}=\tilde {\cal
A}$.

(b) Denote by
\[\sigma_{ij}:=n_{ij}+n_{i,j+1}
+n_{i,j+2}+\cdots\] the number of
chains of length $\ge j$ that end in
$V_i$. Clearly, $k_{i0}=\sigma _{i0}$,
 $k_{i1}=\sigma _{i0}+\sigma
 _{[i+1],1}$, \dots, and
\begin{equation*}\label{fseo}
k_{ij}=\sigma _{i0}+\sigma _{[i+1],1}+
\dots+\sigma _{[i+j],j}
\end{equation*}
for each $1\le i\le t$ and $j\ge 0$.
Therefore,
\[
\sigma _{lj}=k_{ij}-k_{i,j-1},\qquad
l:=[i+j]
\]
(recall that $k_{i,-1}=0$).
 This means that $l\equiv i+j\bmod t$,
$i\equiv l-j\bmod t$, $i=[l-j]$,  and
so
\[
\sigma _{lj}=k_{[l-j],j}-k_{[l-j],j-1}.
\]
We get
\[
n_{lj}=
\sigma _{lj}-\sigma _{l,j+1}=
k_{[l-j],j}-k_{[l-j],j-1}-
k_{[l-j-1],j+1}+k_{[l-j-1],j}.
\tag*{\qedhere}
\]
\end{proof}

\section{Proof of Theorem
\ref{yrw}}

In this section, $\mathbb F=\mathbb C$
or $\mathbb R$.

 (a) Let $\cal A$ and $\cal B$ be
cycles \eqref{ktw} and \eqref{kjtw}.
Let them be topologically equivalent;
that is, $\cal A$ is transformed to $B$
by a system $\{\varphi_i:\mathbb
F^{m_i}\to \mathbb F^{n_i}\}_{i=1}^t$
of homeomorphisms. Let \eqref{groj} be
regularizing decompositions of $\cal A$
and $\cal B$.

First we prove that their regular parts
$\mathcal A_{\mathrm{reg}}$ and
$\mathcal B_{\mathrm{reg}}$ are
topologically equivalent. In notation
\eqref{jyf},
\[\hat{A}_i=A_{[i+t-1]}\cdots
A_{[i+1]}A_i,\quad
\hat{B}_i=B_{[i+t-1]}\cdots
B_{[i+1]}B_i.\] Let $z$ be a
natural number that satisfies both
$\hat A_i^z\mathbb F^{m_i}= \hat
A_i^{z+1}\mathbb F^{m_i}$ and $\hat
B_i^z\mathbb F^{n_i}= \hat
B_i^{z+1}\mathbb F^{n_i}$ for all
$i=1,\dots,t$.
 By
\eqref{kjlj}, the diagram
\[
\xymatrix{ \mathbb
F^{m_i}\ar@{->}^{\hat A^z}[r]
\ar[d]_{\varphi_i}
 &
\mathbb F^{m_i} \ar[d]^{\varphi_i}
           \\
\mathbb F^{n_i}\ar@{->}^{\hat B^z}[r]
 &
\mathbb F^{n_i}
 }
\]
is commutative. Then $\varphi_i\im\hat
A_i^z=\im\hat B_i^z$ for all $i$.
Therefore, the restriction $\hat
\varphi_i:\im\hat A_i^z\to\im\hat
B_i^z$ is a homeomorphism. The system
of homeomorphisms $\hat
\varphi_1,\dots,\hat \varphi_t$
transforms $\tilde {\cal A}$ to $\tilde
{\cal B}$, which are the regular parts
of ${\cal A}$ and ${\cal B}$ by Lemma
\ref{yrwe}(a).

Let us prove that $r=s$, and, after a
suitable renumbering, $\mathcal A_i$
and $\mathcal B_i$ are isomorphic for
all $i=1,\dots,r$. Since all summands
$\mathcal A_i$ and $\mathcal B_i$ with
$i\ge 1$ can be given by chains of
basic vectors, it suffices to prove
that $n_{ij}=n'_{ij}$ for all $i$ and
$j$, where $n'_{ij}$ is the number of
singular summands $\mathcal
B_1,\dots,\mathcal B_s$ in \eqref{groj}
given by chains of length $j$ that end
in the $i$th space $\mathbb F^{n_i}$.

Due to \eqref{byt}, it suffices to
prove that the numbers $k_{ij}$ are
invariant with respect to topological
equivalence.

In the same manner as $k_{ij}$ is
constructed by $\cal A$, we construct
$k_{ij}'$ by $\cal B$. Let us fix $i$
and $j$ and prove that
$k_{ij}=k_{ij}'$. Write
\[
A:=A_{[i+j]}\cdots
      A_{[i+1]}A_i,\quad
      B:=B_{[i+j]}\cdots
      B_{[i+1]}B_i,\qquad q:=[i+j+1]
\]
and consider the commutative diagram
\begin{equation}\label{iug}
\begin{split}
\xymatrix{ \mathbb
F^{m_i}\ar@{->}^{A}[r]
\ar[d]_{\varphi_i}
 &
\mathbb F^{m_q} \ar[d]^{\varphi_q}
           \\
\mathbb F^{n_i}\ar@{->}^{B}[r]
 &
\mathbb F^{n_q}
 }
\end{split}
\end{equation}
which is a fragment of \eqref{kjlj}. We
have
\begin{equation*}\label{org}
k_{ij}=\dim\Ker A=m_i-\dim\im A,\quad
k'_{ij}=n_i-\dim\im B.
\end{equation*}
Because $\varphi_i:\mathbb F^{m_i}\to
\mathbb F^{n_i}$ is a homeomorphism,
$m_i=n_i$ (see \cite[Corollary
19.10]{Bred} or \cite[Section
11]{McCl}). Since the diagram
\eqref{iug} is commutative,
$\varphi_q(\im A)=\im B$. Hence, the
vector spaces $\im A$ and $\im B$ are
homeomorphic, and so $\dim\im A=\dim\im
B$, which proves $k_{ij}=k_{ij}'$.

(b) Each regular cycle $\cal A$ of the
form \eqref{ktw} is isomorphic to the
cycle ${\cal A}'$ of the form
\eqref{tmtw} since the diagram
\begin{equation}\label{iej}
\begin{split}
\xymatrix@R=60pt{ {\mathbb F^{m_1}}
\ar@{->}[d]^{\mathbbm
1} \ar@{<-}@/_2pc/[rrrrr]^{A_t\cdots
A_2A_1} \ar@{->}[r]^{\mathbbm 1}
   & {\mathbb F^{m_2}}
\ar@{->}[d]^{A_1}
\ar@{->}[r]^{\mathbbm 1}
&{\mathbb F^{m_3}}\ar@{->}[d]^{A_2A_1}
   &{\mathbb F^{m_4}}
\ar@{->}[d]^{A_3A_2A_1}
\ar@{<-}[l]_{\mathbbm 1}
\ar@{->}[r]^{\ \mathbbm 1\ }
& {\ \dots\
}\ar@{->}[r]^{\ \ \mathbbm 1}
   &{\mathbb F^{m_t}}
\ar@{->}[d]^{A_{t-1}
\cdots A_2A_1}   \\
\mathbb F^{m_1}
\ar@{<-}@/_2pc/[rrrrr]^{A_t}
\ar@{->}[r]^{ A_1}&
{\mathbb F^{m_2}}\ar@{->}[r]^{A_2}
&{\mathbb F^{m_3}}
&{\mathbb F^{m_4}}
\ar@{<-}[l]_{A_3}
\ar@{->}[r]^{\ A_4\ }
& {\ \dots\
}\ar@{->}[r]^{\ A_{t-1}\ }
   &{\mathbb F^{m_t}}}
\end{split}
\end{equation}
is commutative.

Let $\cal A$ and $\cal B$ be regular
cycles of the form \eqref{ktw} and
\eqref{kjtw}. Let them be topologically
equivalent; that is, $\cal A$ is
transformed to $B$ by a system
$\varphi=(\varphi_1,\dots,\varphi_t)$
of homeomorphisms; see \eqref{kjlj}. By
\eqref{kye},
\begin{align*}
\varphi_1A_{t}A_{t-1}\cdots A_{1}
&=B_{t}\varphi_tA_{t-1}\cdots A_{1}
\\&=B_{t}B_{t-1}\varphi_{t-1}A_{t-2}\cdots
A_{1} =\dots=B_{t}B_{t-1}\cdots
B_{1}\varphi_1,
\end{align*}
and so the cycles ${{\mathbb
F^{m_1}}\!\!\righttoleftarrow\!
{A}_t\cdots{A}_2{A}_1}$ and ${{\mathbb
F^{m_1}}\!\!\righttoleftarrow\!
{B}_t\cdots{B}_2{B}_1}$ are
topologically equivalent via $\varphi
_1$.

Conversely, let ${{\mathbb
F^{m_1}}\!\!\righttoleftarrow\!
{A}_t\cdots{A}_2{A}_1}$ and ${{\mathbb
F^{m_1}}\!\!\righttoleftarrow\!
{B}_t\cdots{B}_2{B}_1}$ be
topologically equivalent via some
homeomorphism $\varphi _1$, and let
${\cal A}'$ and ${\cal B}'$ be
constructed by ${\cal A}$ and ${\cal
B}$ as in \eqref{tmtw}. Then ${\cal
A}'$ and ${\cal B}'$ are topologically
equivalent via the system of
homeomorphisms
$\varphi=(\varphi_1,\varphi_1,\dots,\varphi_1)$.
Let $\varepsilon$ and $\delta$ be
systems of linear bijections that
transform ${\cal A}'$ to ${\cal A}$ and
${\cal B}'$ to ${\cal B}$; see
\eqref{iej}. Then ${\cal A}$ and ${\cal
B}$ are topologically equivalent via
the system of homeomorphisms $\delta
\varphi \varepsilon ^{-1}$.

\end{document}